\documentclass[a4paper,12pt,oneside]{article}

\usepackage{amsmath,amssymb, amsthm, float}
\usepackage[english]{babel}
\usepackage{url}
\usepackage{tikz,pgf}
\usepackage[all]{xy}
\usepackage{tikz-cd}
\newcommand{\gal}{\text{Gal}}
\usetikzlibrary{calc}
\usetikzlibrary{decorations.markings}
\usetikzlibrary{shapes.geometric}
\usetikzlibrary{patterns}
\usetikzlibrary{intersections}
\usetikzlibrary{arrows.meta}
\usepackage[a4paper, total={6in, 8in}]{geometry}
\usetikzlibrary{arrows,positioning,calc,patterns}
\usepackage{tikz-cd}
\usepackage{comment}
\usepackage{hyperref}

\hypersetup{
	colorlinks,
	linkcolor={red!50!black},
	citecolor={blue!50!black},
	urlcolor={blue!80!black},
	breaklinks=true
}
\usepackage{cleveref}
\tikzstyle{vertex}=[circle, draw=white, fill=black, inner sep=0pt, minimum size=4pt]
\tikzstyle{labelsty}=[font=\scriptsize]
\colorlet{ecol}{black!50!white}
\tikzstyle{edge}=[ecol,line width=1.5pt]

\newtheorem{lemma}{Lemma}[section]

\newtheorem{theorem}[lemma]{Theorem}
\newtheorem{remark}[lemma]{Remark}
\theoremstyle{definition}

\newcommand{\C}{\mathbb{C}}
\newcommand{\R}{\mathbb{R}}

\newcommand{\Z}{\mathbb{Z}}
\newcommand{\F}{\mathbb{F}}
\newcommand{\E}{\mathbb{E}}
\newcommand{\K}{\mathbb{K}}
\renewcommand*{\L}{\mathbb{L}}
\newcommand{\g}{\mathcal{G}}

\newcommand{\Gal}{\operatorname{Gal}}

\title{On Galois groups of type-1 minimally rigid graphs}
\author{Mehdi Makhul, Josef Schicho, Audie Warren}

\newcommand{\Addresses}{{
  \bigskip
  \footnotesize

M. Makhul, \textsc{Johann Radon Institute for Computational and Applied Mathematics (RICAM), Linz, Austria}\par\nopagebreak
  \textit{E-mail address}, \texttt{mehdi.makhul@oeaw.ac.at}

  \medskip

  J. Schicho, \textsc{Research Institute for Symbolic Computation (RISC), Linz, Austria}\par\nopagebreak
  \textit{E-mail address}, \texttt{Josef.Schicho@risc.jku.at}

  \medskip

   A. Warren, \textsc{Johann Radon Institute for Computational and Applied Mathematics (RICAM), Linz, Austria}\par\nopagebreak
  \textit{E-mail address}, \texttt{audie.warren@oeaw.ac.at}}
  }

\begin{document}

\maketitle

\begin{abstract}
  For every graph that is minimally rigid in the plane, its Galois group is defined as the Galois group generated by the coordinates of its planar realizations, assuming that the edge lengths are transcendental and algebraically independent. In this paper we compute the Galois group of all minimally rigid graphs that can be constructed from a single edge by repeated Henneberg 1-steps. It turns out that any such group is totally imprimitive, i.e., it is determined by all the partitions it preserves.   
\end{abstract}

\section{Introduction}

A graph is called rigid in the plane if, given generic lengths for its edges, the number of ways to position the vertices in $\mathbb C^2$ which respect these edge lengths is finite up to rotations and translations. A graph is called \textit{minimally rigid} if the removal of any edge causes the graph to become non-rigid. Minimally rigid graphs have been characterised by Geiringer\cite{Geiringer}, and independently by Laman~\cite{Laman}. It is known that all mimimally rigid graphs can be constructed from a single edge by two basic operations on graphs, called the Henneberg 1-step and Henneberg 2-step.

This paper is devoted to the problem of calculating Galois groups for minimally rigid graphs of type-$1$ - that is, minimally rigid graphs that can be constructed from a single edge by repeated applications of the Henneberg 1-step. Galois groups of minimally rigid graphs were studied previously in \cite{OP2007}, where the authors conjectured that if a minimally rigid graph is $3$-connected then its Galois group is not solvable. The minimally rigid graphs that can be constructed by repeated applications of the Henneberg 1-step fall into the subclass of those for which Owen \cite{Owen91} proved that the Galois group is solvable -- actually, it is proven that the Galois group is even a 2-group in this case. As a consequence, the realizations of such a graph for given edge lengths can be constructed by a ruler and compass (see \cite{GaoChou}). The main result of this paper is a complete description of the Galois group for type-$1$ graphs.
 
This question may be embedded into the wider class of problems of computing Galois groups for enumerative geometric problems, such as lines on cubic surfaces. 
For any geometric problem that has only finitely many solutions for a given instance of parameters such a Galois group can be defined. Computing the Galois groups for such geometric problems is an active area in algebraic geometry, with far-reaching results \cite{Esterov2019}, \cite{Harris1979}, \cite{HK2022} \cite{FY2021}. Known methods for computing these Galois groups include monodromy and the decomposition of the discriminant variety (see \cite{Esterov2019}). In this paper, we stay on a level that is more elementary: we just use the Galois correspondence, properties of polynomials, and known results on rigid graphs.

The first sequence in \cite{Sloane} gives the number of isomorphism classes of groups of order $n$, for each $n$. This sequence has high peaks at a power of $2$:
the number of isomorphism classes is particularly large if $n$ is a power of $2$. Therefore there are potentially many candidates for the Galois groups we
want to identify - recalling that it was already proven that the Galois group of a type-1 graph is a $2$-group. We observed a common feature in all these groups, which can be described as follows. The Galois groups come with a transitive action on the
set of realizations. The group is called imprimitive if the set on which it acts can be partitioned into blocks, such that each permutation induces a permutation on blocks. We call a group {\em totally imprimitive} if there is a set of partitions into blocks, such that the group consists of all 
permutations that take blocks into blocks, for each partition. Theorem~\ref{thm: main} says that the Galois group of any minimally rigid graph that can be
constructed by Henneberg 1-steps is totally imprimitive. Any block partition is related to a triple of vertices: two realizations are in the same block if and only if the signed area of triangle given by the triple of vertices is the same for both realizations.

The structure of the paper is as follows. In Section \ref{sec:prelim} we give the necessary definitions and preliminaries. In Section \ref{sec:proof} we prove our main result, and in Section~\ref{ssec:examples} we give an example of how our result can be used to calculate Galois groups.

\section{Preliminaries} \label{sec:prelim}

  Let $\g=(V,E)$ be a graph. A \emph{labelling} of $\g$ is a map $\lambda \colon E \rightarrow \C$. A \emph{realization} of $\g$ is a function $\rho \colon V \rightarrow \C^2$. We say that a realization of $\g$ is compatible with a labelling $\lambda$ if for  each  edge $e\in E$, the squared distance between its endpoints agrees with its label, that is,
\[
\lambda(e)=(x_u-x_v)^2 +(y_u-y_v)^2 \ \ \text{where } e=\left\{u, v \right\}, \rho(u) = (x_u,y_u), \rho(v)=(x_v,y_v).
\]
Labels of the graph therefore correspond to squared edge lengths in the realisation. We will often abuse notation and consider a labelling as a vector of complex numbers, and a realisation as a vector giving the vertex positions. Two realisations of a graph with a fixed labelling are considered equivalent if the points given are the same up to rotation and translation. 

There is a classification (see \cite{Geiringer}, \cite{Laman}) of minimally rigid graphs (also called Laman graphs), in purely combinatorial terms; a graph is minimally rigid iff $|E|= 2|V|-3$, and for every subgraph $\g'=(V',E')$ we have $|E'|\le 2|V'|-3$. For a given minimally rigid graph $\g$, each of its realizations corresponds to a point in $\C^{2|V|}$, and each of its labellings correspond to a point in $\C^{|E|}$. In order to remove the equivalence by rotations and translations, we fix two vertices $v_1,v_2$ that are connected by an edge, and we assume that $\lambda(v_1,v_2)=1$ and that $\rho(v_1)=(0,0)$, $\rho(v_2)=(1,0)$. for given edge lengths, we call the set of such normalised realisations $\mathcal F$.

There are two processes, called the Henneberg steps, which can generate all minimally rigid graphs from an edge.

\begin{theorem} \label{thm: Henneberg}
By the following two rules we can construct all minimally rigid graphs, starting from a single edge.
\begin{itemize}
	\item Henneberg 1-step: add a new vertex and connect it to two
	existing vertices.
	\item Henneberg 2-step: select three vertices of the graph, at
	least two of which are connected by an edge e; delete the edge e;
	add a new vertex and connect it to the three chosen vertices.
\end{itemize}
\end{theorem}
A graph which can be constructed from a single edge by applications of only the Henneberg 1-step is called a \textit{type 1 graph}.
	
\subsection{The Galois group of a minimally rigid graph}
Given a minimally rigid graph $\g$, from hereon we shall always consider a fixed generic labelling $\lambda$ of $\g$. In particular we assume that 
$\lambda(v_1,v_2)=1$, and that the set of all other labels, given by 
$$\Lambda := \{\lambda(e) : e \in E\setminus\{v_1,v_2\}\},$$
generates a purely transcendental field extension of $\mathbb Q$, with $\Lambda$ algebraically independent over $\mathbb Q$; this is the meaning we adopt for a labelling to be called `generic'. We call this field extension $\mathbb K = \mathbb Q (\Lambda)$. For each $\rho \in \mathcal F$, $\rho(V)$ is the set of vertex coordinates given by $\rho$. Then we define the set
$$C:= \bigcup_{\rho \in \mathcal F}\rho(V),$$
which is the set of all vertex coordinates given by some realisation of $\g$ compatible with $\lambda$. We can now define the algebraic (in fact, Galois) field extension $\mathbb E:= \K(C)$. We then have the tower
$$\mathbb Q \subseteq \K \subseteq \E.$$
	The Galois group of $\g$ is then defined as 
 $$\Gal(\g) := \Gal(\E/\K).$$
 The Galois group of $\g$ does not depend on the generic choice of $\lambda$, nor on the choice of vertices embedded at $(0,0)$ and $(1,0)$. $\Gal(\g)$ acts on the set $\mathcal{F}$ in the following way: for a field automorphism $\sigma \in \gal(G)$ and a realisation $\rho \in \mathcal F$, we define
\begin{equation}\label{actiondefinition}
    \sigma(\rho) := ( \sigma(\rho(v_1)), \sigma(\rho(v_2)),...,\sigma(\rho(v_n))
\end{equation}
that is, $\sigma$ is applied to the coordinates of each vertex. Note that since we always fix the vertices $v_1 = (0,0)$, and $v_2 = (1,0)$, and these coordinates lie in the fixed field, they are invariant under the group action. In this sense, an element of the Galois group of a graph corresponds to a permutation of the realisations $\mathcal F$, and we will often abuse notation by regarding $\Gal(\g)$ as both a group of field automorphisms, and as a subgroup of the permutation group $S_{|\mathcal F|}$. 

\subsection{Signed area and the Cayley-Menger determinant}
In order to state our result, we will need the notion of the signed area of a triangle. Given a triple of points $(p_1,p_2,p_3)$ in the real plane, the signed area of $(p_1,p_2,p_3)$ is the area of the triangle given by the three points, multiplied by $+1$ if the vertices are listed in counter-clockwise order, and $-1$ if they are listed in clockwise order. We denote the signed area of $(p_1,p_2,p_3)$ by $a_{p_1p_2p_3}$. This is given as a polynomial relation by
$$a_{p_1p_2p_3} = \frac{1}{2}\left( (x_2-x_3)(y_1-y_3) - (x_1-x_3)(y_2 - y_3) \right)$$
where $p_i=(x_i,y_i)$.
This formula is taken as the definition of $a_{p_1p_2p_3}$ for complex points. Given three vertices $v_i$,$v_j$,$v_k$ of a minimally rigid graph $\g$, and a realisation $\rho$ of $\g$, we define 
$$a_{ijk}(\rho):=a_{\rho(v_i)\rho(v_j)\rho(v_k)}.$$
 
The square of the area $A$ of a triangle with edge lengths $l_1,l_2,l_3 \in \mathbb C$ can be calculated using the Cayley-Menger determinant, which gives the equation
$$-16A^2 = \begin{vmatrix}
0 & l_1^2 & l_2^2 & 1 \\
l_1^2 & 0 & l_3^2 & 1  \\
l_2^2 & l_3^2 & 0 & 1 \\
1 & 1 & 1 & 0
\end{vmatrix}.$$
 The squared area of a triangle therefore fulfills a degree two polynomial in the \emph{squared} edge lengths. 
 \subsection{Type-1 graphs}

 Suppose $\g'$ is a minimally rigid graph on $n-1$ vertices, and $\g$ is constructed from~$\g'$ by a single Henneberg 1-step, which adds an $n'$th vertex connected to two vertices of~$\g'$. It is clear that each realisation of $\g'$ splits into two realisations $\rho_1, \rho_2 \in \mathcal F$ of $\g$, depending on the choice of position for the last vertex; this vertex can always be reflected in the line connecting its two neighbours. We define an equivalence relation on realisations of $\g$ in the following way
$$\rho \sim \rho' \iff \rho|_{n-1} = \rho'|_{n-1}.$$
Where $\rho|_{n-1}$ denotes $\rho$ restricted to the $n-1$ vertices of $\g'$. Each equivalence class has size two, and the number of equivalence classes is equal to the number of realisations of the graph $\g'$, call the set of these realisations $\mathcal F'$. We define a subgroup $H \subseteq S_{|\mathcal F|}$ as follows:
\begin{equation}\label{eqn:H}
    H:= \{ h \in S_{|\mathcal F|} : \forall \rho, \rho' \in \mathcal F,\ \rho \sim \rho' \implies h(\rho) \sim h(\rho') \}.
\end{equation}
Note that we are viewing elements of $S_{|\mathcal F|}$ as acting on $\mathcal F = \{\rho_1,\rho_2,...,\rho_{|\mathcal F|}\}$ itself. We see that $H$ are those permutations of $\mathcal F$ which respect the equivalence relation $\sim$. We claim that $\gal(\g)$ is a subgroup of $H$. Indeed, suppose we take $\sigma \in \gal(\g)$, and two realisations $\rho,\rho' \in \mathcal F$ with $\rho \sim \rho'$. Since they are equivalent, their restriction to the first $n-1$ vertices coincides. Then by the definition of the group action \eqref{actiondefinition}, we have
\begin{align*}
    \sigma(\rho)|_{n-1} &= (\sigma(\rho(v_1)), \sigma(\rho(v_2)),...,\sigma(\rho(v_{n-1})) \\&= (\sigma(\rho'(v_1)), \sigma(\rho'(v_2)),...,\sigma(\rho'(v_{n-1})) = \sigma(\rho')|_{n-1}
\end{align*}
and therefore $\sigma(\rho)\sim \sigma(\rho')$, as needed. Furthermore, there is a group homomorphism $\psi :H \rightarrow S_{|\mathcal F'|}$ which maps each permutation $h \in H$ to the permutation in $S_{|\mathcal F'|}$ given by the action of $h$ on each equivalence class. Note that the equivalence classes correspond precisely to realisations of the smaller graph $\g'$, which we write as $\mathcal{F}'=\{c_1,c_2,...,c_{|\mathcal F'|}\}$, with $c_i = \{\rho_i,\rho_i'\}$, and define 
$$\psi(h)(c_i) := \{h(\rho_i),h(\rho_i')\} = c_j, \ \text{for some }j.$$
Note that if the permutation $h \in H$ is actually a Galois group element $\sigma \in \Gal(\g)$, then the definition of $\psi$ can be equivalently written as
$$\psi(\sigma)(v_1,v_2,...,v_{n-1}) = (\sigma(v_1),\sigma(v_2),...,\sigma(v_{n-1})).$$ 
We can now state our main result, firstly in its inductive (and more general) form.
\begin{theorem} \label{thm: main}
For each minimally rigid graph $\g$ which is constructed from a minimally rigid subgraph $\g'$ by a single Henneberg 1-step which creates a new vertex $v_n$ and connects it to vertices $v_i,v_j$, we have
\begin{align*}
    \Gal(\g) \cong \{ h \in H : \psi(h) \in \Gal(\g'), \forall \rho, \rho' \in \mathcal F, \ &a_{ijn}(\rho) = a_{ijn}(\rho')  \\ & \implies a_{ijn}(h\rho) = a_{ijn}(h\rho')\}.
\end{align*}
    
\end{theorem}
Alternatively, if the graph $\g$ is itself a type-1 graph, our main result can be stated in a direct form. Since $\g$ is a type-1 graph, there exists a sequence $m_1,m_2,...,m_{n-2}$ of type-1 moves which construct $\g$ from a single edge. Each move $m_l$ consists of an ordered triple of vertices of $\g$, that is, $m_l = (i_l,j_l,n_l)$, with $n_l$ the newly constructed vertex and $i_l,j_l$ the two previously constructed vertices which are then connected to $n_l$. With this notation, Theorem \ref{thm: main} implies the following.
\begin{theorem} \label{thm: first version}
Let $\g$ be a minimally rigid type-1 graph with $n$ vertices and construction steps $(m_l)_{l=1}^{n-2}$ as defined above. Then we have
\begin{align*}
\Gal(\g) \cong\{ g \in S_{|\mathcal{F}|} : \forall l=1,\dots, n-2, \forall \rho, \rho' \in \mathcal F, \ &a_{i_lj_ln_l}(\rho) = a_{i_lj_ln_l}(\rho')  \\ & \implies a_{i_lj_ln_l}(g \rho) = a_{i_lj_ln_l}(g \rho')\}.
\end{align*}
\end{theorem}

\textbf{Proof that Theorem \ref{thm: main} implies Theorem \ref{thm: first version}}

Let us denote by $G_1$ and $G_2$ the corresponding groups given in Theorem \ref{thm: main} and \ref{thm: first version} respectively, for some type-1 graph $\g$ on $n$ vertices. 

We begin by showing that $G_2 \subseteq G_1$, which is done by induction - we assume that for graphs with up to $n-1$ vertices, this inclusion is satisfied. For $g \in G_2$, we wish to show that $g \in H$, and furthermore that $\psi(g) \in \g'= \g|_{n-1}$ (Note that the third condition required is trivially satisfied by all elements of $G_2$). Take $\rho, \rho' \in \mathcal F$ such that $\rho \sim \rho'$, that is, $\rho|_{n-1} = \rho'|_{n-1}$. Since $\g$ is a type-1 graph, each realisation $\rho$ of $\g$ corresponds to a choice of sign for the signed area of the triangle $(i_l,j_l,n_l)$ constructed at each stage - informally speaking, this is a choice of whether to flip the triangle up or down. Since we have $\rho|_{n-1} = \rho'|_{n-1}$, each choice up to the $n-1$'th vertex must have been the same, that is, we have $a_{i_lj_ln_l}(\rho) = a_{i_lj_ln_l}(\rho')$ for each $l=1,...,n-3$. Since $g \in G_2$, each of these equalities also holds for $g\rho$ and $g \rho'$ - meaning that the same sign choices have been made in the construction of these two realisations, again for $l=1,...,n-3$. But then we have $(g\rho)|_{n-1} = (g\rho')|_{n-1}$, and therefore $g\rho \sim g\rho'$, proving that $g \in H$.

We now wish to prove that $\psi(g) \in \Gal(\g')$. Note that since we have already proved $g \in H$, $\psi$ can indeed be applied to $g$. It is in this step that we use induction; let $G_2'$ be the corresponding group for the smaller graph $\g'$, which by induction is contained in $\Gal(\g')$. It is therefore enough to prove that $\psi(g) \in G_2'$, meaning that we need to prove for each pair $\mu,\mu' \in \mathcal F'$, and $1\leq l \leq n-3$, we have $a_{i_lj_ln_l}(\mu) = a_{i_lj_ln_l}(\mu')  \implies a_{i_lj_ln_l}(\psi(g) \mu) = a_{i_lj_ln_l}(\psi(g) \mu')$.

Since $\mu \in \mathcal F'$, there exists some $\rho \in \mathcal F$ such that $\rho|_{n-1} = \mu$. Similarly, there is some $\rho' \in \mathcal F$ such that $\rho'|_{n-1} = \mu'$. In fact there are two such choices - it does not matter which we pick. We now have that $\psi(g)(\mu) = (g\rho)|_{n-1}$, and $\psi(g)(\mu') = (g\rho')|_{n-1}$. We then have, for each $1 \leq l \leq n-3$, the chain of implications
\begin{align*}
    a_{i_lj_ln_l}(\mu) = a_{i_lj_ln_l}(\mu') & \implies a_{i_lj_ln_l}(\rho) = a_{i_lj_ln_l}(\rho') \\
    & \implies a_{i_lj_ln_l}(g\rho) = a_{i_lj_ln_l}(g\rho') \\
    & \implies a_{i_lj_ln_l}((g\rho)|_{n-1}) = a_{i_lj_ln_l}((g\rho')|_{n-1}) \\
    & \implies a_{i_lj_ln_l}(\psi(g)(\mu)) = a_{i_lj_ln_l}(\psi(g)\mu'),
\end{align*}
as needed. This completes the proof that $G_2 \subseteq G_1$.

To show that $ \Gal(\g) \cong G_1 \subseteq G_2 $, we simply note that the conditions defining $G_2$ are given by polynomial relations which are defined over the base field $\mathbb K$, with vertex coordinates as variables. Since the Galois group $\Gal(\g)$ must preserve polynomial relations (meaning that if $f$ is a polynomial defined over $\mathbb K$ in variables from $\mathbb E$, then the image of any zero of $f$ under an element of the Galois group of $\mathbb E / \mathbb K$ must also be a zero of $f$), we see that $G_1 \subseteq G_2$.

\section{Proof of Theorem \ref{thm: main}}\label{sec:proof}
In this section we prove Theorem \ref{thm: main}. We retain all of the notation from above.
\begin{proof}
We begin by defining the set 
\begin{align*}
    J := \{ h \in H : \psi(h) \in \Gal(\g'), \forall \rho, \rho' \in \mathcal F, \ &a_{ijn}(\rho) = a_{ijn}(\rho')  \\ & \implies a_{ijn}(h\rho) = a_{ijn}(h\rho')\}.
\end{align*}
Our aim is to show that $\Gal(\g) = J$. Note that $J$ is a subgroup of $H$. 
\begin{lemma} \label{claim:subgroup}
We have $\Gal(\g) \subseteq J$.
\end{lemma} 
\begin{proof}For $\sigma \in \Gal(\g)$, we have already seen that $\Gal(\g) \subseteq H$, so that $\sigma \in H$. Secondly, we want to show that $\psi(\sigma) \in \Gal(\g')$. For this we note that $\psi(\sigma)$ is a permutation of the realisations $\mathcal{F}'$ of the smaller graph, and as we have seen above, it is given by applying the field automorphism $\sigma$ to each vertex coordinate. This permutation then corresponds to the field automorphism $\sigma$ restricted to the intermediate field $\K \subseteq \F \subseteq \E$ containing the vertex entries for only the smaller graph $\g'$, and is therefore an element of $\Gal(\g')$. Lastly, each $\sigma \in \g$ fulfils the condition of preserving the relation of equal signed area. Indeed, signed area can be given by a polynomial equation, and elements of the Galois group preserve polynomial relations.
\end{proof}
In the next step we prove that $|\Gal(\g)| = |J|$. We do this by making use of the map $\psi:H \rightarrow S_{|\mathcal F'|}$. Indeed, we prove the following.
\begin{lemma} We have the following two equalities.
\begin{itemize}\label{claim:psi}
\item $\psi(\Gal(\g)) = \psi(J)$.
    \item $|\Gal(\g) \cap \ker(\psi)| =|J \cap \ker(\psi)|$.
\end{itemize}
\end{lemma}
Given Lemma \ref{claim:psi}, an application of the first isomorphism theorem to the following two group homomorphisms given by restrictions of $\psi$,
$$\psi_{J}:J \rightarrow \psi(J), \quad \psi_{\Gal(\g)}:\Gal(\g) \rightarrow \psi(\Gal(\g)),$$
yields $|J| = |\Gal(\g)|$, which together with Lemma \ref{claim:subgroup} finishes the proof.

\subsection{Proof that $\psi(\Gal(\g)) = \psi(J)$} \label{sec:chain}

We now prove the first assertion in Lemma \ref{claim:psi}. Note that we have
$$\psi(\Gal(\g)) \subseteq \psi(J) \subseteq \psi(\{h \in H : \psi(h) \in \Gal(\g')\}) \subseteq \Gal(\g').$$
We show that $ \Gal(\g') \subseteq \psi(\Gal(\g))$, closing the chain of inclusions above. In order to do this, we use the fact that for any tower of Galois extensions
$$\K \subseteq \F \subseteq \E,$$
a field automorphism in $\Gal(\F/\K)$ can be extended to a field automorphism in $\Gal(\E/ \K)$. Suppose there exists a field automorphism $\sigma' \in \Gal(\g')$ such that $\nexists \sigma \in \Gal(\g)$ such that $\psi(\sigma) = \sigma'$. However, as argued above, the map $\psi$ corresponds to a restriction of the field automorphism $\sigma \in \Gal(\g)$ to the intermediate field $\F$ given by adjoining to $\K$ all vertex coordinates of all realisations of $\g'$. We then have a tower of Galois extensions $\K \subseteq \F \subseteq \E$, and $\sigma' \in \Gal(\F/ \K)$ which does not extend to a field automorphism $\sigma \in \Gal(\E/ \K)$, giving a contradiction. Therefore we have $ \Gal(\g') \subseteq \psi(\Gal(\g))$. We then conclude that $\psi(\Gal(\g)) \subseteq \psi(J) \subseteq \psi(\Gal(\g))$, and therefore the sets are equal.
\subsection{Proof that $|\Gal(\g) \cap \ker(\psi)| =|J \cap \ker(\psi)|$}
We now prove that $|\Gal(\g) \cap \ker(\psi)| =|J \cap \ker(\psi)|$. In order to do this, we define the natural number $k$ as the number of distinct square distances which occur between the vertices $i$ and $j$ among all realisations of $\g'$. That is,
$$k := \left| \left\{ ||\rho(i) - \rho(j)||^2 : \rho \in \mathcal{F}' \right\}\right|.$$
Recall that $i,j$ are the base points of the Henneberg 1-step used to construct $\g$. Furthermore, we define $\lambda_{i,j}(\rho)$ to be the distance between the vertices $i$ and $j$ in a realisation $\rho$ of $\g'$. The integer $k$ is therefore the number of distinct values of $\lambda_{i,j}$ found among all realisations of $\g'$. We apply this definition equally to realisations of the larger graph $\g$ - that is, $\lambda_{i,j}$ applied to a realisation of $\g$ gives the squared distance between $i$ and $j$ in that realisation. We aim to show that both $|\Gal(\g) \cap \ker(\psi)|$ and $|J \cap \ker(\psi)|$ have cardinality equal to $2^{k}$.
\subsubsection{Proof that $|J \cap \ker(\psi)|=2^k$}
We first show that $|J \cap \ker(\psi)|=2^k$. In order to do this, we partition the realisations $\mathcal{F}$ of $\g$ in terms of the squared distance between $i$ and $j$, that is, two realisations $\rho$ and $\rho'$ are in the same part if $\lambda_{i,j}(\rho) = \lambda_{i,j}(\rho')$. We call this partition $\mathcal{P}_1$, and note that it has $k$ parts. We further refine this partition by splitting each part into two subsequent parts, in terms of the signed area of the triangle given by vertices $i,j,n$. That is, we define a partition $\mathcal P_2$ of $\mathcal F$ such that $\rho$ and $\rho'$ are in the same part iff $\lambda_{i,j}(\rho) = \lambda_{i,j}(\rho')$ and $a_{ijn}(\rho) = a_{ijn}(\rho')$. We note that each part $S \in \mathcal P_1$ contains precisely two parts, call them $S^+$ and $S^-$, of $\mathcal P_2$, depending on the sign of the signed area; the signed area is equally distributed in this way because each realisation can be paired with its reflection, which preserves $\lambda_{i,j}$ but reverses the signed area of the triangle. We analyse how $\psi$ interacts with these partitions of $\mathcal F$. 

We claim that an element $\alpha \in J \cap \ker(\psi)$ satisfies
$$\rho \in S \in \mathcal P_1 \implies \alpha(\rho) \in S,$$
that is, $\alpha$ preserves the parts of $\mathcal P_1$. Indeed, since $\alpha \in \ker(\psi)$, $\psi(\alpha)$ acts as the identity on $\mathcal F'$. In particular, we must have $\alpha(\rho) \in \{\rho,\rho'\}$ where $\rho'|_{n-1} = \rho|_{n-1}$. Since $\alpha(\rho)|_{n-1} = \rho|_{n-1}$, we have $\lambda_{i,j}(\rho) = \lambda_{i,j}(\alpha(\rho))$, as needed.

Secondly, we claim that for a part $S = S^+ \sqcup S^-$ of $\mathcal P_1$ and $\alpha \in J \cap \ker(\psi)$, either $\alpha(\rho) = \rho$ for all $\rho \in S$, or $\alpha(S^+)= S^-$ and $\alpha(S^-) = S^+$. In words, this means that $\alpha$ can do one of two things to $S$; it can fix everything, or swap every realisation $\rho$ with its equivalent realisation under $\sim$. We emphasize that each equivalence class $\{\rho,\rho'\}$ is contained in a single $S \in \mathcal P_1$, and precisely one element of the class is in $S^+$, the other being in $S^-$. This is because equivalent realisations have signed areas which are of opposite sign.

To prove this claim, we take WLOG an equivalence class $\{\rho, \rho'\}$, with $\rho \in S^+$. Assume that $\alpha(\rho) = \rho$. Note that therefore $\alpha(\rho') = \rho' \in S^-$. Since $\alpha \in J$, we have that 
$$a_{ijn}(\rho) = a_{ijn}(\rho') \implies a_{ijn}(\alpha(\rho)) = a_{ijn}(\alpha(\rho')),$$
and therefore for any other realisation $\rho_0 \in S^+$, we must have $ a_{ijn}(\rho_0)= a_{ijn}(\alpha(\rho_0))$, by applying the above implication to $\rho$ and $\rho_0$. Therefore $\alpha(\rho_0) = \rho_0$. Since we can give the same argument for $S^-$, we conclude that for all $\rho_0 \in S$, $\alpha(\rho_0) = \rho_0$. Using a similar argument for the case $\alpha(\rho) = \rho'$ yields the second case, where each realisation is swapped with the other element in its equivalence class, and so $\alpha(S^+)= S^-$ and $\alpha(S^-) = S^+$, as claimed.

Therefore, in order to define $\alpha$, for each part $S \in \mathcal P_1$ we need only decide whether $\alpha$ swaps every realisation in $S$ with its equivalent realisation under $\sim$, or if $\alpha$ pointwise fixes all of $S$. Since we get two choices for each part $S \in \mathcal P_1$, we have $|J \cap \ker(\psi)| \leq 2^k$. We now have to show that $|J \cap \ker(\psi)| \geq 2^k$. To do this, we define $2^k$ permutations $\beta:\mathcal F \rightarrow \mathcal F$ as follows. Let $I \subseteq \{1,2,...,k\}$ be an index set of the parts $S_1,S_2,...,S_k$ of $\mathcal{P}_1$. Note that there are $2^k$ choices for $I$. We then define the permutation
$$\beta_{I}(\rho) = \begin{cases}\rho \text{ if } \rho \in S_i, \ i \in I \\ \rho' \text{ if } \rho \in S_i, \ i \notin I.
\end{cases}$$
Where $\rho$ is in the equivalence class $\{\rho,\rho'\}$. Note that difference choices of $I$ give different permutations, and therefore the set of all $\beta_{I}$ is a set of $2^k$ distinct permutations. We prove that for all such $\beta_{I}$, we have $\beta_I \in J \cap \ker(\psi)$. 

To do this, we first note that $\beta_I \in H$, and further $\beta_I \in \ker(\psi)$. Indeed, by definition, $\beta_I(\rho)$ always remains within the class $\{\rho,\rho'\}$, and $\beta_I$ is therefore in $\ker(\psi)$. 

We now check whether $\beta_I \in J$. Since $\psi(\beta_I)$ is the identity permutation, we also have $\psi(\beta_I) \in \Gal(\g')$. The final property to check is preservation of equality of signed area. Suppose we take $\rho_1$ and $\rho_2$ with $a_{ijn}(\rho_1) = a_{ijn}(\rho_2)$. Since they have the same signed area, there exists some $i \in \{1,...,k\}$ such that $\rho_1,\rho_2 \in S_i$. Note that here we have used the algebraic independence of the squared distances, which implies that two realisations with the same signed area $a_{ijn}$ must have the same squared distance $\lambda_{1,2}$. Since these realisations lie in the same $S_i$, $\beta_I$ acts on them in the same way, that is, either $\beta_I(\rho_1) =\rho_1$ and $\beta_I(\rho_2) = \rho_2$, or $\beta_I(\rho_1) =\rho_1'$ and $\beta_I(\rho_2) = \rho_2'$. In either case the signed areas are the same after $\beta_I$; in the first case the signed areas do not change, and in the second they are negated. Therefore, $\beta_I \in J \cap \ker(\psi)$, and so we conclude that $|J \cap \ker(\psi)| = 2^k$.
\subsubsection{Proof that $|\Gal(\g) \cap \ker(\psi)| = 2^k$}\label{sec:distances}
We now prove that $|\Gal(\g) \cap \ker(\psi)| = 2^k$. Let $\sigma \in \Gal(\g)$ be such that $\psi(\sigma)$ is the identity permutation. We have already seen via the chain of inclusions in Section \ref{sec:chain} that $\psi(\Gal(\g)) = \Gal(\g')$, so that considered as a field automorphism, $\psi(\sigma)$ is the identity map from $\F$ to $\F$, where we recall that $\F$ is the field given by extending the field of all squared edge lengths $\K$ by the vertex coordinates of $\g'$. Therefore, $\sigma$ is a field automorphism $\E \rightarrow \E$ which fixes the field $\F$, that is, $\Gal(\g) \cap \ker(\psi) = \Gal(\E / \F)$ (note that $\E / \F$ is a Galois extension since $\E / \K$ is a Galois extension). Since $|\Gal(\E / \F)| = [\E : \F]$, it is enough to give the degree of this field extension. 

The field $\E$ is generated from $\F$ by the vertex coordinates of the last vertex $v_n$, through all realisations of $\g$. For a realisation $\rho \in \mathcal F = \{ \rho_1,...,\rho_{|\mathcal F|}\}$, let $v_n(\rho)$ be the vertex coordinates of $v_n$ in that realisation. We can write the tower of extensions
$$\F=: \F_0 \subseteq \F_1 \subseteq \F_2 \subseteq ... \subseteq \F_{|\mathcal F|} = \E,$$
where $\F_{t+1} := \F_{t}(v_n(\rho_{t+1}))$. We claim that for all $t=0,1,...,|\mathcal F|-1$, we have $[\F_{t+1} : \F_t] \leq 2$. To prove this, we first show that $\F_{t+1} \subseteq \F_t(a_{ijn}(\rho_{t+1}))$, and that $a_{ijn}(\rho_{t+1})$ gives an extension of $\F_t$ of degree at most $2$. To prove that $\F_{t+1} \subseteq \F_t(a_{ijn}(\rho_{t+1}))$, we show that both vertex coordinates of $v_n(\rho_{t+1})$ can be expressed as a polynomial in terms of $\Delta := a_{ijn}(\rho_{t+1})$. Using the notation $v_i = (x_i,y_i), v_j = (x_j,y_j), v_n = (x_n,y_n)$, we have the three equations
\begin{align*}
    \Delta = \frac{1}{2}(x_iy_j -& x_iy_n + x_jy_n - x_jy_i + x_ny_i - x_ny_j)\\
    \lambda_{i,n} &= (x_i-x_n)^2 + (y_i-y_n)^2\\
    \lambda_{j,n} &= (x_j-x_n)^2 + (y_j-y_n)^2.
\end{align*}
 The last two equations allow us to express $x_n$ linearly in $y_n$, and therefore from the first equation, both $x_n$ and $y_n$ can be expressed in the field $\F_t(\Delta)$. Secondly, from the Cayley-Menger determinant, $\Delta^2$ can be written as a polynomial in the squared distnces between the vertices $v_i,v_j,v_n$, which can all be calculated within the base field $\F_t$. Therefore, $\F_{t}(\Delta)$ is an extension of degree at most $2$.
 
 With this in hand, we define a subset $T \subseteq 
\{0,1,...,|\mathcal F|-1\}$ of indices where we have $[\F_{t+1}:\F_t] = 2$. By the tower law, we then have $[\E:\F] = 2^{|T|}$. We claim that for $t,t' \in I$, we must have $\lambda_{i,j}(\rho_{t+1}) \neq \lambda_{i,j}(\rho_{t'+1})$.

To prove this, we note that from the previous argument, we must have $\F_{t+1} = \F_t(\Delta_{t+1})$, where we define $\Delta_{t+1} := a_{ijn}(\rho_{t+1})$. Suppose that for some $t,t' \in T$, we have $\lambda_{i,j}(\rho_{t+1}) = \lambda_{i,j}(\rho_{t'+1})$. Note that this implies $\Delta_{t+1} = \pm \Delta_{t'+1}$, since all edge lengths of the triangle $v_iv_jv_n$ are the same. But then the two field extensions $\F_{t+1}/\F_t$ and $\F_{t'+1}/\F_{t'}$ cannot both be degree two, since they can both be generated by the same element, say $\Delta_{t+1}$, and there is an inclusion of base fields. We therefore have that $[\E : \F] \leq 2^k$. Note that there is a one-to-one correspondence between all distances $\lambda_{i,j}$ and all possible (non-signed) areas of the triangle $v_iv_jv_n$ among all realisations in $\mathcal{F}$. Let us call the set of all such areas $A = \{\alpha_1,\alpha_2,...,\alpha_k\}$. Since each extension above is generated by such an area, we have 
$$\E = \F(A).$$
In order to prove that $[\E :\F] = 2^k$, we give the following lemma. 
\begin{lemma}\label{lm: not-square1}
Let $\L$ be a field not of characteristic two, and let $a_1,\dots, a_k \in \L$. For each $s=1,...,k$, let $\alpha_s^2=a_s$, where $\alpha_s\in \overline{\L}$. If we have that for each subset $I \subseteq \left\{1,\dots, k \right\}$, the product
\[
\prod_{s\in I}a_s,
\]
is not a square in $\L$, then we have
\[
[ \L(\alpha_1,\dots,\alpha_k): \L]=2^k.
\]
\end{lemma}
To prove this lemma, we require the following auxiliary lemma.
\begin{lemma} \label{lem:aux}
  Let $\L$ be a field not of characteristic two, and let $a_1,a_2,...,a_k \in \L$, such that for each $s=1,...,k$ we have $a_s = \alpha_s^2$ for some $\alpha_s \in \overline{\L}$. Suppose that $[\L(\alpha_1,...,\alpha_k):\L] = 2^k$, and that $\gamma \in \L(\alpha_1,...,\alpha_k)\setminus \L$ is an element such that $\gamma^2 \in \L$. Then there exists a set $I \subseteq \{1,...,k\}$ such that 
  $$\gamma \prod_{s \in I}\alpha_s \in \L.$$
\end{lemma}
\begin{proof}
    We begin with the base case $k=1$, in which case we have a single element $a \in \L$ such that $a = \alpha^2$, and $[\L(\alpha):\L]=2$. By assumption, we are given an element $\gamma \in \L(\alpha) \setminus \L$ with $\gamma^2 \in \L$. We can then write $$\gamma = c_1 + c_2\alpha \implies \gamma^2 = c_1^2 + 2c_1c_2\alpha + c_2^2a_1$$
    for $c_1,c_2 \in \L$. Since $\gamma^2 \in \L$, we have $c_1c_2 = 0$. We cannot have $c_2=0$ since this implies $\gamma \in \L$. If $c_1=0$, then we have $\gamma = c_2\alpha$, and therefore the product $\gamma \alpha = c_2 a \in \L$, as needed.

    We now go to the inductive step. We assume the lemma is true for $k \leq n$. We have the elements $a_1,a_2,...,a_{n+1} \in \L$, and their roots $\alpha_s$ in the algebraic closure. We have that $[\L(\alpha_1,...,\alpha_{n+1}):\L]=2^{n+1}$, and we are given an element $\gamma \in \L(\alpha_1,...,\alpha_{n+1}) \setminus \L $ such that $\gamma^2 \in \L$. By considering the extension 
    $$\L(\alpha_1,...,\alpha_{n+1}) / \L(\alpha_1,...,\alpha_{n}),$$
    there exist $c_1,c_2 \in \L(\alpha_1,...,\alpha_{n})$ such that $\gamma = c_1 + c_2\alpha_{n+1}$. Since $\gamma^2 \in \L$, and $\gamma^2 = c_1^2 + 2c_1c_2 \alpha_{n+1} + c_2^2a_{n+1},$ we must have $c_1c_2=0$. First suppose $c_1=0$, in which case $\gamma = c_2\alpha_{n+1}$, and so $\gamma^2 = c_2^2 a_{n+1} \in \L$, we find that $c_2^2 \in \L$. If $c_2 \notin \L$, then we apply the induction hypothesis to $\L(\alpha_1,...,\alpha_{n})$ with the element $c_2$, implying the existence of a set $I' \subseteq \{1,...,n\}$ such that $c_2 \prod_{s \in I'} \alpha_s \in \L$. Define $I = I' \cup \{n+1\}$. We then have
    $$\gamma \prod_{s \in I}\alpha_s = c_2\alpha_{n+1} \prod_{s \in I} \alpha_s = c_2 a_{n+1}\prod_{s \in I'} \alpha_s \in \L,$$
    as needed. If, on the other hand, $c_2 \in \L$, then we have $\gamma =c_2\alpha_{n+1}$, and so the product $\gamma \alpha_{n+1} = c_2 a_{n+1} \in \L$, and so the set $I =\{n+1\}$ satisfies the lemma.

    In the second case we have $c_2=0$, and therefore $\gamma = c_1$. In this case we are again done by the induction hypothesis applied to the element $c_1 \in \L(\alpha_1,...,\alpha_n) \setminus \L$. Note that $c_1 \notin \L$, as we would then have $\gamma = c_1 \in \L$, which contradicts our assumptions.
\end{proof}

\begin{proof}[Proof of Lemma \ref{lm: not-square1}]
We prove the contrapositive, that is, if 
\[
[\L(\alpha_1,\dots,\alpha_k): \L]<2^k,
\]
 then there exists a subset $I \subseteq\left\{1,\dots, k \right\}$ such that $\Pi_{s\in I}a_s$ is a square in $\L$.
The proof is by induction. The base case $k=1$ is trivial and corresponds to the statement that $[\L(\alpha_1):\L] =1$, so that $\alpha_1\in \L$, and so $a_1 = \alpha_1^2$ is a square in $\L$. Suppose that the lemma holds for $k\leq n$. We show that it also holds for $k=n+1$. Consider the field extension 
$$\L(\alpha_1,\dots,\alpha_n,\alpha_{n+1}) / \L(\alpha_1,\dots,\alpha_n).$$
We consider the degree of this extension. We have $$[\L(\alpha_1,\dots,\alpha_n,\alpha_{n+1}): \L(\alpha_1,\dots,\alpha_n)][ \L(\alpha_1,\dots,\alpha_n):\L]<2^{n+1}.$$ If we were to have $[ \L(\alpha_1,\dots,\alpha_n):\L] < 2^{n}$, then the proof would be complete by the induction hypothesis, as this implies the existence of a subset $I \subseteq \{1,2,...,n\} \subseteq \{1,2,...,n+1\}$ with the property that $\Pi_{s \in I} a_s$ is a square in $\L$. Therefore, we can assume $[ \L(\alpha_1,\dots,\alpha_n):\L] = 2^{n}$, so that we must have
$$[\L(\alpha_1,\dots,\alpha_n,\alpha_{n+1}): \L(\alpha_1,\dots,\alpha_n)] =1.$$
If this is the case, then clearly we have $\alpha_{n+1} \in \L(\alpha_1,\dots,\alpha_n).$ If we have $\alpha_{n+1} \in \L$, then we would have that $a_{n+1} = \alpha_{n+1}^2$ is a square in $\L$ and we are done by taking $I=\{n+1\}$. Therefore, we may apply Lemma~\ref{lem:aux} to this field, with the choice $\gamma= \alpha_{n+1}$. This implies the existence of a set $I' \subseteq \{1,...,n\}$ such that we have 
$$\alpha_{n+1}\prod_{s \in I'} \alpha_s \in \L.$$
Defining the set $I := I' \cup \{n+1\}$, we now have that 
$$\prod_{s \in I} a_s = \left( \alpha_{n+1} \prod_{s \in I'} \alpha_s\right)^2$$
showing that this product is a square in $\L$, as needed.
\end{proof}
We now wish to apply Lemma \ref{lm: not-square1} to the field $\F$, with the elements $a_s = \alpha_s^2$ being squared areas from the set $A$. In order to show that $[\mathbb E: \mathbb F] = 2^k$, we need to show that for each subset $I \subseteq \{1,...,k\}$, the product 
$$\prod_{s \in I} a_s$$
is not a square in $\mathbb F$. Using the Cayley-Menger formula, each $a_s$ can be written as
\begin{equation}\label{eq:CM}
-16a_s = \lambda_{i,j,s}^2 - 2\lambda_{i,j,s}(\lambda_{i,n} + \lambda_{j,n}) + (\lambda_{j,n} - \lambda_{i,n})^2,
\end{equation}
where $\lambda_{i,j,s}$ is the squared distance between $v_i$ and $v_j$ in a realisation with the square area of $v_iv_jv_n$ equal to $a_s$. Note that $\lambda_{i,n}$ and $\lambda_{j,n}$ were fixed by the choice of $\Lambda$ at the beginning of the proof. We consider the right hand side of \eqref{eq:CM} as a polynomial in two variables $X,Y$, corresponding to $\lambda_{i,n}$ and $\lambda_{j,n}$. We define the field $\mathbb K' := \mathbb Q ( \lambda(e) : e \in E(\g'))$. Note that since $\Lambda$ is algebraically independent over $\mathbb Q$, the squared edge lengths $\lambda_{i,n}$ and $\lambda_{j,n}$ are algebraically independent over $\mathbb K'$. We further define a field $\mathbb E':= \mathbb K'(\{ \text{ vertex coordinates of }\g'\})$. As $\mathbb E'$ is an algebraic extension of $\K'$, $\lambda_{i,n}$ and $\lambda_{j,n}$ are algebraically independent over $\mathbb E'$. We also note that $\mathbb F = \mathbb E'(\lambda_{i,n},\lambda_{j,n})$, so that in fact
$$\mathbb F \cong \text{Frac}(\mathbb E'[X,Y]).$$
We denote by $\phi$ the isomorphism between $\mathbb F$ and the fraction field of the polynomials $\mathbb E'[X,Y]$. Recall that we aim to show that the product $\prod_{s \in I} a_s$
is not a square in $\mathbb F$. Suppose that it were a square; that is, there exists some $r \in \mathbb F$ such that $$\prod_{s \in I} a_s = r^2.$$
By \eqref{eq:CM}, this implies
$$\prod_{s \in I} \left(\lambda_{i,j,s}^2 - 2\lambda_{i,j,s}(\lambda_{i,n} + \lambda_{j,n}) + (\lambda_{j,n} - \lambda_{i,n})^2\right)= (-16)^{|I|}r^2.$$
Applying $\phi$ to each side and a slight rearrangement then gives
\begin{equation}
    \label{eq:polyprod} \prod_{s \in I} \left(-\lambda_{i,j,s}^2 + 2\lambda_{i,j,s}(X + Y) - (Y - X)^2\right)= (4^{|I|}\phi(r))^2.
\end{equation}
Note that the left hand side of \eqref{eq:polyprod} is a product of distinct polynomials. Since $\phi(r)^2$ is itself a polynomial, we must have that $\phi(r)$ is a polynomial. Therefore, $\prod_{s \in I} \left(-\lambda_{i,j,s}^2 + 2\lambda_{i,j,s}(X + Y) - (Y - X)^2\right)$ must be a perfect square in the polynomial ring $\mathbb E'[X,Y]$. For this to occur, at least one of the factors in \eqref{eq:polyprod} must be reducible. We claim that this does not happen; in fact, we prove that the polynomial $\lambda_{i,j,s}^2 - 2\lambda_{i,j,s}(X + Y) + (Y - X)^2$ is irreducible in $\mathbb E'[X,Y]$ for all $s$. We perform the simple change of variables $X+Y \rightarrow Z$ and $X-Y \rightarrow W$, giving $\lambda_{i,j,s}^2 - 2\lambda_{i,j,s}Z + W^2$. If this polynomial factorises, it must factor into linear pieces. If we write
$$\lambda_{i,j,s}^2 - 2\lambda_{i,j,s}Z + W^2 = (W + c_1Z + c_2)(W + c_3Z + c_4),$$
we see that we must have $c_1c_3=0$, so without loss of generality assume that $c_1 = 0$. As no term $WZ$ appears, we must have $c_3 =0$. But now no term of $Z$ is present, giving a contradiction. Therefore, the product $\prod_{s \in I} a_s$ is not a square in $\mathbb F$, as needed. Applying Lemma \ref{lm: not-square1} then gives 
$$[\E:\F] = 2^k.$$
\end{proof}
\section{A simple example applying Theorem \ref{thm: main}}\label{ssec:examples}
In this section, we will make use of Theorem \ref{thm: main} to determine the Galois group of the following minimally rigid graph $\g$.
\begin{figure}[H]
    \centering
    \begin{tikzpicture}[scale=1]
\node[vertex, label={1}] (1) at (5,0) {};
\node[vertex, label={3}] (2) at (5.6,1) {};
\node[vertex, label={2}] (3) at (6.5,0) {};
\node[vertex, label={4}] (4) at (5.6,-1) {};
\node[vertex, label={5}] (5) at (8,0.4) {};

\draw[edge] (1)edge(2) (1)edge(3) (2)edge(3) (4)edge(1) (4)edge(3) (5)edge(2) (5)edge(4);
\end{tikzpicture}
\caption{The graph $\g$}
\end{figure}

It is not hard to see that the Galois group of the minimally rigid subgraph $\g' \subset \g$ corresponding to vertices $1,2,3,4$ is isomorphic to $\Z_2^2$. We will show that Theorem~\ref{thm: main} implies that
\[
\Gal(\g) = D_4 \times \Z_2.
\] 
The graph $\g$ has eight realisations, which we label as $\mathcal{F}=\left\{\rho_1, \dots, \rho_8\right\}$. Note that for any realisation of $\g$, we can find another realisation $\rho'$ which is given by reflecting only vertex $5$ in the line through $3$ and $4$, and that this pair of realisations form an equivalence class under the relation $\sim$. Without loss of generality, we assume that the equivalence classes are $\{\rho_{i}, \rho_{i+1}\}$ for $i=1,3,5,7$, where $\rho_i$ for $i$ odd is chosen to be the realisation such that the signed area $a_{345}(\rho_i)$ is positive. Theorem \ref{thm: main} then tells us the following.
\begin{align*}
    \Gal(\g) \cong \{ h \in H : \psi(h) \in \mathbb \Gal(\g'),\ \forall \rho, \rho' \in \mathcal F, \ &a_{345}(\rho) = a_{345}(\rho')  \\ & \implies a_{345}(h\rho) = a_{345}(h\rho')\},
\end{align*}
where we recall that 
$$H:= \{ h \in S_8 : \forall \rho, \rho' \in \mathcal F,\ \rho \sim \rho' \implies h(\rho) \sim h(\rho') \},$$
and that $\psi(h)$ is the induced permutation in $S_4$ given by $h$ acting on the equivalence classes $\{\rho_i,\rho_{i+1}\}$ for $i=1,3,5,7$. Each equivalence class corresponds to a certain realisation of the subgraph $\g'$, and we choose our labelling in the following way:
\begin{itemize}
    \item $\rho_1$ and $\rho_2$ correspond to the realisation of $\g'$ where both triangles $123$ and $124$ have positive orientation. Call this class $A$.
    \item $\rho_3$ and $\rho_4$ correspond to the realisation of $\g'$ where the triangle $123$ has positive orientation, and $124$ has negative orientation. Call this class $B$.
    \item $\rho_5$ and $\rho_6$ correspond to the realisation of $\g'$ where the triangle $123$ has negative orientation, and $124$ has positive orientation. Call this class $C$.
    \item $\rho_7$ and $\rho_8$ correspond to the realisation of $\g'$ where both triangles $123$ and $124$ have negative orientation. Call this class $D$.
\end{itemize}
It is now simple to write down which elements of $S_4$ are in the Galois group $\Gal(\g')$; we get the four permutations (in cycle notation)
$$\Gal(\g') = \{(A)(B)(C)(D), (AB)(CD), (AC)(BD),(AD)(BC)\}.$$

In addition to the condition $\psi(h) \in \Gal(\g')$, we also need $h$ to preserve the equivalence classes given by the signed area $a_{345}$. These equivalence classes are the pairs
$$\{\rho_1, \rho_7\}, \{\rho_2, \rho_8\}, \{ \rho_3, \rho_5\}, \{\rho_4, \rho_6\}.$$
Using Theorem \ref{thm: main}, one can verify that the following elements of $S_8$ belong to $\Gal(\g)$.
\[
h_1=(1423)(5768),\quad h_2=(34)(56),\quad h_3=(18)(27)(36)(45).
\]
Moreover, they fulfil the following equalities
\[
h_1^4=h_2^2=h_3^2=e, \quad h_2h_1h_2=h_1^{-1},\quad h_3h_1=h_1h_3,\quad h_2h_3=h_3h_2.
\]
This shows that the direct product $D_4 \times \Z_2$, which is isomorphic to the group $\left<h_1, h_2,h_3\right>$, is a subgroup of $\Gal(\g)$. On the other hand, we have $|\Gal(\g)|=16$. Indeed, consider the following tower of fields 
\[
\K \subseteq \F \subseteq \E,
\]
where $\K$ is $\mathbb Q$ extended by the squared edge lengths of $\g$, $\F$ is $\K$ extended by all vertex coordinates of realisations of the subgraph $\g'$, and $\E$ extends $\F$ by vertex coordinates of all realisations of $\g$ (in particular, by the coordinates of the final vertex $5$). We therefore have $\Gal(\g')=\Gal(\g)/ \Gal(\E/\F)$. As we have seen in Section \ref{sec:distances}, the degree of $\mathbb E$ over $\mathbb F$ is $2^k$, where $k$ is the number of distinct distances between two non-edge vertices of $\g$; in our example, $k=2$. Thus, $|\Gal(\g)|=|\Gal(\g')||\Gal(\E/\F)| = 16$, and hence
\[
\Gal(\g)\cong D_4 \times \Z_2.
\]

\begin{remark} \rm
The Galois group gives restrictions on the possible number of real realizations for real labellings. In this case, the field $\K$ is a 
subfield of $\R$. Since $\E$ is a subfield of $\C$ closed under all automorphisms that fix $\K$, the restriction of complex conjugation
to $\E$ is a field automorphism of $\E$ fixing $\K$, that is, an element $\epsilon$ of the Galois group. Note that $\epsilon^2$ is the
identity. A compatible realization is real if and only if it is fixed by $\epsilon$. If $k$ is the number of real realizations, then 
there must exist an element in the Galois group of order two or one that fixes exactly $k$ elements.

In the example above, any involution in $\Gal(\g)$ has either no fixed points at all or exactly four fixed points. Hence the number 
of real solutions for a generic real labelling of $\g$ is either zero, four, or eight.
\end{remark}

\section*{Acknowledgments}
The first listed author was supported by the Austrian Science Fund FWF Project P33003. We thank Niels Lubbes and Matteo Gallet for helpful conversations.

\bibliography{galoisgroupminimallyrigid}{}

\bibliographystyle{amsplain}

\Addresses
	
\end{document}